\documentclass[12pt]{amsart}
\usepackage[margin=1.1in]{geometry}
\usepackage{amssymb,amsfonts,amsmath,mathrsfs,bbm,constants}
\usepackage[unicode]{hyperref}
\usepackage[capitalise]{cleveref}
\usepackage[shortlabels]{enumitem}

\hypersetup{colorlinks=true, citecolor=blue, linkcolor=blue, urlcolor=blue, pdfstartview=FitH, pdfauthor=Gergely Harcos and Jesse Thorner, pdftitle=Tatuzawa's theorem for Rankin--Selberg $L$-functions}

\newtheorem{theorem}{Theorem}[section]
\newtheorem{corollary}[theorem]{Corollary}

\newtheorem{proposition}[theorem]{Proposition}
\newtheorem{lemma}[theorem]{Lemma}

\newtheorem{question*}{Question}
\newtheorem{problem*}{Problem}
\newtheorem*{hypothesis*}{Hypothesis}
\newtheorem*{definition*}{Definition}

\theoremstyle{definition}

\theoremstyle{remark}
\newtheorem{remark}[theorem]{Remark}

\numberwithin{equation}{section}

\crefname{figure}{Figure}{Figures}
\theoremstyle{plain}
\newtheorem*{theorem*}{Theorem}
\newtheorem*{lemma*}{Lemma}
\crefname{theorems}{Theorem}{Theorems}
\crefname{corollaries}{Corollary}{Corollaries}
\newtheorem*{corollary*}{Corollary}
\crefname{corollaries*}{Corollary}{Corollaries}
\crefname{lemma}{Lemma}{Lemmata}
\crefname{proposition}{Proposition}{Propositions}
\crefname{conjectures}{Conjecture}{Conjectures}
\newtheorem*{conjecture*}{Conjecture}
\crefname{conjectures*}{Conjecture}{Conjectures}
\crefname{definitions}{Definition}{Definitions}
\crefname{hypotheses}{Hypothesis}{Hypotheses}

\renewcommand{\tilde}{\widetilde}
\renewcommand{\bar}{\overline}
\renewcommand{\epsilon}{\varepsilon}

\renewcommand{\Re}{\mathrm{Re}}
\renewcommand{\Im}{\mathrm{Im}}
\newcommand{\re}{\Re}
\newcommand{\im}{\Im}
\newcommand{\GL}{\mathrm{GL}}
\newcommand{\N}{\mathrm{N}}

\newcommand{\diag}{\mathrm{diag}}
\newcommand{\sgn}{\mathrm{sgn}}

\newcommand{\A}{\mathbb{A}}
\newcommand{\Z}{\mathbb{Z}}
\newcommand{\Q}{\mathbb{Q}}
\newcommand{\R}{\mathbb{R}}
\newcommand{\CC}{\mathbb{C}}
\newcommand{\ka}{\mathfrak{a}}
\newcommand{\kp}{\mathfrak{p}}
\newcommand{\kq}{\mathfrak{q}}
\newcommand{\cL}{\mathcal{L}}
\newcommand{\cO}{\mathcal{O}}

\newconstantfamily{abcon}{symbol=c}

\title{Tatuzawa's theorem for Rankin--Selberg $L$-functions}

\author{Gergely Harcos}
\address{Alfr{\'e}d R{\'e}nyi Institute of Mathematics, POB 127, Budapest H-1364, Hungary}
\email{\href{mailto:gharcos@renyi.hu}{gharcos@renyi.hu}}

\author{Jesse Thorner}
\address{Department of Mathematics, University of Illinois, Urbana, IL 61801, USA}
\email{\href{mailto:jesse.thorner@gmail.com}{jesse.thorner@gmail.com}}

\begin{document}

\begin{abstract}
Let $\pi$ and $\pi'$ be unitary cuspidal automorphic representations of $\GL(n)$ and $\GL(n')$ over a number field $F$. We establish a new zero-free region for all $\GL(1)$-twists of the Rankin--Selberg $L$-function $L(s,\pi\times\pi')$, generalizing Tatuzawa's refinement of Siegel's work on Dirichlet $L$-functions.  As a corollary, we show that for all $\epsilon>0$, there exists an effectively computable constant $c>0$ depending only on $(n,n',[F:\Q],\epsilon)$ such that $L(s,\pi\times\pi')$ has at most one zero (necessarily simple) in the region
\[
\re(s)\geq 1-c/(C(\pi)C(\pi')(|\im(s)|+1))^{\epsilon},
\]
where $C(\pi)$ and $C(\pi')$ are the analytic conductors.  A crucial component of our proof is a new standard zero-free region for any twist of $L(s,\pi\times\tilde{\pi})$ by an idele class character $\chi$ apart from a possible single exceptional zero (necessarily real and simple) that can occur only when $\pi\otimes\chi^2=\pi$.  This extends earlier work of Humphries and Thorner.
\end{abstract}

\thanks{GH was supported by the MTA--HUN-REN RI Lend{\"u}let Automorphic Research Group and NKFIH (National Research, Development and Innovation Office) grant K~143876. JT was partially supported by the National Science Foundation (DMS-2401311) and the Simons Foundation (MP-TSM-00002484).}


\maketitle

\section{Introduction and the main result}

This paper continues the work in \cite{HarcosThorner} on zero-free regions for Rankin--Selberg $L$-functions.  Let $F$ be a number field with absolute discriminant $D_F$ and adele ring $\A_F$. By a unitary cuspidal ($L^2$-)automorphic representation $\pi$ of $\GL_n(\A_F)$, we mean an irreducible closed invariant subspace of $L^2_\mathrm{cusp}(\GL_n(F)\backslash\GL_n(\A_F),\omega)$ under the right regular representation, where $\omega=\omega_\pi$ is a unitary Hecke character called the central character of $\pi$. Let $\mathfrak{F}_n$ be the set of these representations $\pi$.  Given $\pi\in\mathfrak{F}_n$, let $\tilde{\pi}\in\mathfrak{F}_n$ be the contragredient, and let $L(s,\pi)$ be the standard $L$-function of $\pi$.  Let $C(\pi)\geq 3$ denote the analytic conductor of $\pi$.  Given $(\pi,\chi)\in\mathfrak{F}_n\times\mathfrak{F}_1$, let $\pi\otimes\chi$ be the (unique) element of $\mathfrak{F}_n$ whose space consists of the functions $f(x)\chi(\det x)$, where $f(x)$ is from the space of $\pi$. It is straightforward to check that $\pi\otimes\chi$ is isomorphic to the representation $g\mapsto\pi(g)\chi(\det g)$. For future convenience, as in \cite{HarcosThorner}, we introduce $\mathfrak{F}_n^*$, the set of $\pi\in\mathfrak{F}_n$ for which $\omega_{\pi}$ is trivial on the diagonally embedded positive reals.  For each $\pi\in\mathfrak{F}_n$, there exists a unique pair $(\pi^*,t_{\pi})\in\mathfrak{F}_n^*\times\R$ such that
\[
\pi=\pi^*\otimes|\cdot|^{it_{\pi}},\qquad L(s,\pi) = L(s+it_{\pi},\pi^*).
\]

Given $(\pi,\pi')\in\mathfrak{F}_n\times\mathfrak{F}_{n'}$, let $L(s,\pi\times\pi')$ be the associated Rankin--Selberg $L$-function, whose basic properties were established by Jacquet, Piatetski-Shapiro, and Shalika~\cite{JPSS,JS1,JS2}.  If $(\pi,\pi')\in\mathfrak{F}_n^*\times\mathfrak{F}_{n'}^*$, then $L(s,\pi\times\pi')$ is holomorphic away from a possible pole at $s=1$, which occurs if and only if $\pi'=\tilde{\pi}$.  If $\pi'\in\mathfrak{F}_1^*$ is the trivial representation $\mathbbm{1}$, then $L(s,\pi\times\pi')=L(s,\pi)$.

Jacquet and Shalika~\cite{JacquetShalika} proved that if $\pi\in\mathfrak{F}_n$, then $L(s,\pi)\neq 0$ for $\re(s)\geq 1$.  Recently, Wattanawanichkul~\cite[Theorem~1.1]{Wattanawanichkul} (see also \cite[Appendix]{Humphries}) proved that $L(\sigma+it,\pi^*)$ has at most one zero in the region
\begin{equation}
\label{eqn:standard_ZFR_log_1}
\sigma\geq 1-\frac{1}{16(2n+3)\log(C(\pi^*)(|t|+3)^{n[F:\Q]})}.
\end{equation}
If the exceptional zero exists, then it is necessarily real and simple, and $\pi^*=\tilde{\pi}^*$. This directly generalizes classical results for Dirichlet $L$-functions.  We could have equivalently presented a zero-free interval for $L(\sigma,\pi)$, but the presentation with $L(\sigma+it,\pi^*)$ helps to clarify the properties of the exceptional zero.  In certain cases, the exceptional zero has been precluded entirely; see \cite{Banks,HoffsteinLockhart,HoffsteinRamakrishnan,Thorner_Siegel1}.

Shahidi~\cite{Shahidi} proved that if $(\pi,\pi')\in\mathfrak{F}_n\times\mathfrak{F}_{n'}$, then $L(s,\pi\times\pi')\neq 0$ for $\re(s)\geq 1$.  Recently, Wattanawanichkul~\cite[Theorems~1.2 and~1.3]{Wattanawanichkul} (see also \cite[Appendix]{Humphries} and \cite{HumphriesThorner}) proved that if
\begin{equation}
\label{eqn:standard_ZFR_conds}
L(s,\pi^*\times\pi'^*)=L(s,\tilde{\pi}^*\times\tilde{\pi}'^*)\qquad\textup{or}\qquad\pi'^*=\tilde{\pi}'^*,
\end{equation}
then $L(\sigma+it,\pi^*\times\pi'^*)$ has at most one zero in the region
\begin{equation}
\label{eqn:standard_ZFR_log}
\sigma\geq 1-\frac{1}{66(n+n')\log(C(\pi^*)C(\pi'^*)(|t|+3)^{n[F:\Q]})}.
\end{equation}
If the exceptional zero exists, then it is must be real and simple, and $L(s,\pi^*\times\pi'^*)=L(s,\tilde{\pi}^*\times\tilde{\pi}'^*)$. In certain cases, the exceptional zero in \eqref{eqn:standard_ZFR_log} can be precluded even when \eqref{eqn:standard_ZFR_conds} is not satisfied; see \cite{Luo,RamakrishnanWang,Thorner_Siegel1,Thorner_Siegel2}.  For further historical remarks, see \cite[Section~1]{HarcosThorner} and the references contained therein.

There exist zero-free regions with no exceptions that hold for all Rankin--Selberg $L$-functions, but they are weaker than \eqref{eqn:standard_ZFR_log}.  The two strongest results are as follows.  First, Brumley (\cite{Brumley}, \cite[Appendix]{Lapid}) proved that if $(\pi,\pi')\in\mathfrak{F}_n\times\mathfrak{F}_{n'}$, then for all $\epsilon>0$, there exists an effectively computable constant $\Cl[abcon]{Brumley}=\Cr{Brumley}(n,n',F,\epsilon)>0$ such that
\begin{equation}
\label{eqn:Brumley_ZFR}
L(\sigma,\pi\times\pi')\neq 0,\qquad \sigma \geq 1-\Cr{Brumley}/(C(\pi)C(\pi'))^{n+n'-1+\epsilon}.
\end{equation}
(See Zhang~\cite{Zhang} for a recent numerical improvement in the exponent of \eqref{eqn:Brumley_ZFR} in the $|\cdot|^{it}$-twist aspect.)  Second, the authors~\cite{HarcosThorner} proved that for all $\epsilon>0$, there exists an ineffective constant $\Cl[abcon]{Crelle}=\Cr{Crelle}(\pi,\pi',\epsilon)>0$ such that if $\chi\in\mathfrak{F}_1$, then
\begin{equation}
\label{eqn:ineffective}
L(\sigma,\pi\times(\pi'\otimes\chi))\neq 0,\qquad \sigma\geq 1-\Cr{Crelle}C(\chi)^{-\epsilon}.
\end{equation}
When $\pi=\pi'=\mathbbm{1}$ and $F=\Q$, \eqref{eqn:ineffective} recovers a landmark result of Siegel~\cite{Siegel} for $L$-functions of quadratic Dirichlet characters: If $\mathcal{Q}$ is the set of primitive quadratic Dirichlet characters, then for all $\epsilon>0$, there exists an {\it ineffective} constant $\Cl[abcon]{Siegel}=\Cr{Siegel}(\epsilon)>0$ such that
\begin{equation}
\label{eqn:Siegel}
L(1,\chi)\geq \Cr{Siegel}q_{\chi}^{-\epsilon},\qquad\chi\in\mathcal{Q}.
\end{equation}

Tatuzawa~\cite{Tatuzawa} refined Siegel's ineffective result \eqref{eqn:Siegel} by proving that there exists an {\it effectively computable} constant $\Cl[abcon]{Tatuzawa}=\Cr{Tatuzawa}(\epsilon)>0$ such that
\begin{equation}
\label{eqn:Tatuzawa}
\text{$L(1,\chi)\geq \Cr{Tatuzawa} q_{\chi}^{-\epsilon}$ holds for all $\chi\in\mathcal{Q}$ with one possible exception.}
\end{equation}
In fact \cite[Theorem~1]{Tatuzawa} states that $\Cr{Tatuzawa}=\epsilon/10$ is permissible.  This narrowly misses an effective version of \eqref{eqn:Siegel}, which, for example, would significantly impact the distribution of primes in arithmetic progressions and the computation of class numbers of imaginary quadratic fields.  Using classical techniques in the theory of Dirichlet $L$-functions, one can prove that \eqref{eqn:Tatuzawa} is equivalent to the following statement:
\begin{equation}
\label{eqn:Tatuzawa_ZFR}
\parbox{0.85\linewidth}{For all $\epsilon>0$, there exists an effectively
computable constant
$\Cl[abcon]{Tatuzawa1}=\Cr{Tatuzawa1}(\epsilon)>0$\\and
$\chi_{\epsilon}\in\mathcal{Q}$ such that if
$\chi\in\mathcal{Q}-\{\chi_{\epsilon}\}$ and $\sigma\geq
1-\Cr{Tatuzawa1}q_{\chi}^{-\epsilon}$, then $L(\sigma,\chi)\neq 0$.}
\end{equation}

Since \eqref{eqn:ineffective} is true, one might hope to prove a version of \eqref{eqn:Tatuzawa} for the $\GL_1$-twists of $L(s,\pi)$ or $L(s,\pi\times\pi')$.  When $F=\Q$, Ichihara and Matsumoto~\cite{IM} proved that if $L(s)$ is an $L$-function satisfying certain restrictive hypotheses and $\epsilon>0$, then there exists an effectively computable constant $\Cl[abcon]{IM}=\Cr{IM}(L,\epsilon)>0$ and a character $\psi=\psi_{L,\epsilon}\in\mathcal{Q}$ such that if $\chi\in\mathcal{Q}-\{\psi\}$, then the twisted $L$-function $L_{\chi}(s)$ satisfies $|L_{\chi}(1)|\geq \Cr{IM}q_{\chi}^{-\epsilon}$. We highlight several desiderata that are not addressed in \cite{IM}.  First, \cite{IM} assumes that $L_{\chi}(s)$ restricted to $|\im(s)|\leq 1$ has a classical zero-free region.  As described above, this hypothesis remains a significant open problem in the context of Rankin--Selberg $L$-functions. Second, \cite{IM} does not determine the effective dependence of $\Cr{IM}$ on the original $L$-function $L(s)$.  Third, like all preceding manifestations of Tatuzawa's ideas, only exceptional behavior of quadratic twists of $L(s)$ are considered.  It is important to have a version of Tatuzawa's theorem for all $\GL_1$-twists of $L(s)$, not just those coming from $\mathcal{Q}$. For example, even if $L(s)$ is not self-dual, there might exist $\chi\in\mathfrak{F}_1-\mathcal{Q}$ such that $L_{\chi}(s)$ is self-dual (and hence might have an exceptional zero). Furthermore, finite order idele class characters only form a small portion of $\mathfrak{F}_1$. In particular, by handling characters of the form $|\cdot|^{it}$, one might be able to use Tatuzawa's ideas to establish an {\it effective} Siegel-type zero-free region apart from at most one exception.  This would be a significant step towards improving Brumley's narrow zero-free region in \eqref{eqn:Brumley_ZFR}.

Let $(\pi,\pi')\in\mathfrak{F}_n\times\mathfrak{F}_{n'}$.  In this paper, we prove a version of \eqref{eqn:Tatuzawa_ZFR} that holds unconditionally for all $\GL_1$-twists of $L(s,\pi)$ and $L(s,\pi\times\pi')$.

\begin{theorem}
\label{thm:Tatuzawa}
Let $(\pi,\pi',\chi)\in\mathfrak{F}_n\times\mathfrak{F}_{n'}\times\mathfrak{F}_1$ and $\epsilon>0$. There exist an effectively computable constant $\Cl[abcon]{ZFR}=\Cr{ZFR}(n,n',[F:\Q],\epsilon)>0$ and a character $\psi=\psi_{\pi,\pi',\epsilon}\in\mathfrak{F}_1$ such that if $L(s,\pi\times(\pi'\otimes\chi))$ differs from $L(s,\pi\times(\pi'\otimes\psi))$, then
\begin{equation}
\label{ZFRfinal}
L(\sigma,\pi\times(\pi'\otimes\chi))\neq 0,\qquad \sigma\geq 1-\Cr{ZFR}(C(\pi)C(\pi')C(\chi))^{-\epsilon}.
\end{equation}
Moreover, $L(s,\pi\times(\pi'\otimes\psi))$ has at most one zero (necessarily simple) in the interval $\sigma\geq 1-\Cr{ZFR}(C(\pi)C(\pi')C(\psi))^{-\epsilon}$.
\end{theorem}

\begin{remark}
There might exist several characters $\chi\in\mathfrak{F}_1$ resulting in the same twisted $L$-function $L(s,\pi\times(\pi'\otimes\chi))$, which is why we consider exceptions at the level of twisted $L$-functions rather than characters.
\end{remark}

\begin{remark}
\cref{thm:Tatuzawa} strengthens \eqref{eqn:ineffective} considerably.  It is new even when $\pi'=\mathbbm{1}$, that is, for the $\GL_1$-twists of the standard $L$-function $L(s,\pi)$.
\end{remark}

By specializing \cref{thm:Tatuzawa} to the shift characters $\chi=|\cdot|^{it}\in\mathfrak{F}_1$ we obtain the following corollary, which significantly improves upon \eqref{eqn:Brumley_ZFR} apart from at most one exceptional zero.

\begin{corollary}
\label{cor:Tatuzawa}
Let $(\pi,\pi')\in\mathfrak{F}_n\times\mathfrak{F}_{n'}$ and $\epsilon>0$. There exists an effectively computable constant $\Cl[abcon]{ZFR3}=\Cr{ZFR3}(n,n',[F:\Q],\epsilon)>0$ such that $L(\sigma+it,\pi\times\pi')$ has at most one zero (necessarily simple) in the region $\sigma\geq 1-\Cr{ZFR3}(C(\pi)C(\pi')(|t|+1))^{-\epsilon}$.
\end{corollary}

\begin{remark}
If the exceptional zero $\beta+i\gamma$ exists, then the much coarser results of Brumley~\cite{Brumley,Lapid} and Zhang~\cite{Zhang} provide the only nontrivial effective upper bound for $\beta$.  This is similar to the best known effective zero-free region for $L$-functions of primitive quadratic Dirichlet characters.
\end{remark}

Our proof deviates substantially from those in \cite{IM,Tatuzawa}.  For example, these works rely on Estermann's approach to a lower bound for $L(1,\chi)$ \cite{Estermann} (see also \cite[\S20]{Davenport}) combined with assumed zero-free regions.  Our proofs rely on Goldfeld's approach to a lower bound for $L(1,\chi)$ \cite{Goldfeld} combined with other techniques to avoid assuming any unproven zero-free regions. In particular, we make a systematic use of the ideas of Goldfeld, Hoffstein, Lieman, Lockhart, and Ramakrishan, originally developed in \cite{HoffsteinLockhart,HoffsteinRamakrishnan}.

We also substantially depart from our earlier work \cite{HarcosThorner}, which relied on the following strategy.  Let $\mathfrak{F}_1^{(j)}=\{\chi\in\mathfrak{F}_1\colon {\chi^*}^j=\mathbbm{1}\}$.  First, let $\chi\in\mathfrak{F}_1^{(1)}$.  We establish the zero-free region \eqref{eqn:ineffective} assuming that there exists $\lambda\in\mathfrak{F}_1^{(1)}$ such that $L(s,\pi\times(\pi'\otimes\lambda))$ has a real zero close to $s=1$ and the $L$-functions
\begin{equation}
\label{eqn:three_twists}
L(s,\pi\times(\pi'\otimes\lambda)),\qquad L(s,\pi\times(\pi'\otimes\chi)),\qquad L(s,\pi\times(\pi'\otimes\chi^2\bar{\lambda}))
\end{equation}
are entire.  If one of the $L$-functions in \eqref{eqn:three_twists} has a pole, then $L(s,\pi\times(\pi'\otimes\lambda))$ has a pole, which allows us to conclude \eqref{eqn:ineffective} for all $\chi\in\mathfrak{F}_1^{(1)}$.  Second, let $\chi\in\mathfrak{F}_1^{(2)}$.  We establish \eqref{eqn:ineffective} assuming that there exists $\lambda\in\mathfrak{F}_1^{(2)}$ such that $L(s,\pi\times(\pi'\otimes\lambda))$ has a real zero close to $s=1$ and the $L$-functions in \eqref{eqn:three_twists} are entire.  If one of the $L$-functions in \eqref{eqn:three_twists} has a pole, then either $L(s,\pi\times(\pi'\otimes\lambda))$ has a pole or $\chi$ lies in one of $O_{\pi,\pi',\epsilon}(1)$ cosets of the group $\mathfrak{F}_1^{(1)}$.  Therefore, we can conclude \eqref{eqn:ineffective} for all $\chi\in\mathfrak{F}_1^{(2)}$.  Third, let $\chi\in\mathfrak{F}_1$.  We establish \eqref{eqn:ineffective} assuming that there exists $\lambda\in\mathfrak{F}_1$ such that $L(s,\pi\times(\pi'\otimes\lambda))$ has a real zero close to $s=1$ and the $L$-functions in \eqref{eqn:three_twists} are entire.  If one of the $L$-functions in \eqref{eqn:three_twists} has a pole, then either $L(s,\pi\times(\pi'\otimes\lambda))$ has a pole or $\chi$ lies in one of $O_{\pi,\pi',\epsilon}(1)$ cosets of the group $\mathfrak{F}_1^{(2)}$.  Therefore, we can conclude \eqref{eqn:ineffective} for all $\chi\in\mathfrak{F}_1$.

Here, we eschew this incremental strategy in favor of a more direct approach, one where the role of the ``semi-exceptional'' twist $L(s,\pi\times(\pi'\otimes\lambda))$ is to pin down the ``exceptional'' twist $L(s,\pi\times(\pi'\otimes\psi))$.  Namely, we prove that if $L(s,\pi\times(\pi'\otimes\lambda))$ has a real zero close to $s=1$ and we restrict to the situation where the $L$-functions in \eqref{eqn:three_twists} are entire, then the possible violations of \eqref{ZFRfinal} correspond bijectively to zeros of $L(s,\pi\times(\pi'\otimes\lambda))$ very close to $s=1$.  However, there is at most one such zero of $L(s,\pi\times(\pi'\otimes\lambda))$, verifying the bulk of \cref{thm:Tatuzawa}. In fact, we eliminate the third $L$-function from the list \eqref{eqn:three_twists} by the following simple yet crucial observation:  If $L(s,\pi\times(\pi'\otimes\chi^2\bar{\lambda}))$ has a pole, then there exists a unique $t\in\R$ such that $\pi'\otimes\chi^2\bar{\lambda}=\tilde{\pi}\otimes|\cdot|^{it}$.  In order to leverage this observation, we require a refinement of the main nonnegativity argument that leads to the zero-free regions in \eqref{eqn:standard_ZFR_log_1} and \eqref{eqn:standard_ZFR_log}.  This is \cref{lem:GHLnew} below, which actually serves multiple purposes in the paper.  As part of our proof of \cref{thm:Tatuzawa}, we use \cref{lem:GHLnew} to establish a new zero-free region for all $\GL_1$-twists of $L(s,\pi\times\tilde{\pi})$.

\begin{theorem}
\label{thm:standard_region}
If $(\pi,\chi)\in\mathfrak{F}_n\times\mathfrak{F}_1^*$, then $L(\sigma+it,\pi\times(\tilde\pi\otimes\chi))$ has at most one zero (necessarily real and simple) in the region
\begin{equation}
\label{eqn:ZFR_standard_region}
\sigma\geq 1-\frac{1}{903\log(C(\pi)^{2n}C(\chi)^{n^2}(|t|+3)^{n^2[F:\Q]})}.
\end{equation}
If the exceptional zero exists, then $\pi\otimes\chi^2=\pi$.  Moreover, if $(\pi,\chi_1,\chi_2)\in\mathfrak{F}_n\times\mathfrak{F}_1^*\times\mathfrak{F}_1^*$ satisfies $\pi\otimes\chi_1\neq\pi\otimes\chi_2$, then $L(s,\pi\times(\tilde{\pi}\otimes\chi_1))L(s,\pi\times(\tilde{\pi}\otimes\chi_2))$ has at most one zero (counted with multiplicity) in the interval
\begin{equation}
\label{twobetasinterval}
\sigma\geq 1-\frac{1}{158\log(C(\pi)^{4n}C(\chi_1)^{n^2}C(\chi_2)^{n^2})}.
\end{equation}
\end{theorem}

\begin{remark}
\cref{thm:standard_region} is completely unconditional.  It was proved (with inexplicit constants) in the special cases where $\chi=\mathbbm{1}$ \cite[Theorem~1.1]{HumphriesThorner} or $F=\Q$ and the conductor of $\chi$ is coprime to the conductor of $\pi$ \cite[Theorem~8.2]{HumphriesThorner2}.
\end{remark}

\begin{remark}
\label{conductorremark}
For the proof of \cref{thm:Tatuzawa}, we only need the first part of \cref{thm:standard_region}, and only when $\pi\otimes\chi^2\neq\pi$.  In this special case, our proof reveals that in \eqref{eqn:ZFR_standard_region}, we can replace $C(\chi)^{n^2}(|t|+3)^{n^2[F:\Q]}$ by the smaller, more canonical quantity $C(it,\chi)^{n^2}$ (see \eqref{eqn:AC_def_standard} below).
\end{remark}

\section{Properties of $L$-functions}

Let $F$ be a number field with adele ring $\A_F$. Let $D_F$ be the absolute discriminant of $F$, $\cO_F$ the ring of integers of $F$, and $\N=\N_{F/\Q}$ the norm defined on nonzero ideals $\ka$ of $\cO_F$ by $\N\ka=|\cO_F/\ka|$. For a place $v$ of $F$, let $v\mid\infty$ (resp. $v\nmid\infty$) denote that $v$ is archimedean (resp. non-archimedean), and let $F_v$ be the corresponding completion of $F$. Each $v\nmid\infty$ corresponds with a prime ideal $\kp$ of $\cO_F$.

\subsection{Standard $L$-functions}

Let $\mathfrak{F}_n$ be the set of cuspidal automorphic representations $\pi$ of $\GL_n(\A_F)$ whose central character $\omega_\pi$ is unitary, and let $\mathfrak{F}_n^*$ be the subset of those $\pi$'s for which $\omega_\pi$ is trivial on the diagonally embedded positive reals. Every $\pi\in\mathfrak{F}_n$ can be written uniquely as
\begin{equation}
\label{eqn:pidecomp}
\pi=\pi^*\otimes|\cdot|^{it_{\pi}},\qquad \pi^*\in\mathfrak{F}_n^*,\quad t_{\pi}\in\R.
\end{equation}

If $\pi\in\mathfrak{F}_n^*$, then for each place $v$, there exists an irreducible admissible representation $\pi_v$ of $\GL_n(F_v)$, with $\pi_v$ ramified for at most finitely many $v$, such that $\pi$ is a restricted tensor product $\otimes_v \pi_v$. When $v\nmid\infty$ and $\kp$ corresponds with $v$, then we write $\pi_v$ and $\pi_{\kp}$ interchangeably. For each prime ideal $\kp$, there exist $n$ Satake parameters $\alpha_{1,\pi}(\kp),\dotsc,\alpha_{n,\pi}(\kp)\in\CC$ such that the standard $L$-function $L(s,\pi)$ of $\pi$ is the absolutely convergent Euler product
\[
L(s,\pi)=\prod_{\kp}L(s,\pi_{\kp}) = \prod_{\kp}\prod_{j=1}^n \frac{1}{1-\alpha_{j,\pi}(\kp)\N\kp^{-s}},\qquad\Re(s)>1.
\]
For $v\mid \infty$, define
\[
\Gamma_v(s)=\begin{cases}
\pi^{-s/2}\Gamma(s/2)&\mbox{if $F_v=\R$,}\\
2(2\pi)^{-s}\Gamma(s)&\mbox{if $F_v=\CC$.}
\end{cases}
\]
There exist $n$ Langlands parameters $\mu_{1,\pi}(v),\dotsc,\mu_{n,\pi}(v)\in\CC$ such that
\begin{equation}\label{eq:Lpiv}
L(s,\pi_{v})=\prod_{j=1}^n \Gamma_v(s+\mu_{j,\pi}(v)).
\end{equation}

Let $\kq_{\pi}$ denote the conductor of $\pi$, and let $\mathbbm{1}\in\mathfrak{F}_1^*$ denote the trivial representation (whose $L$-function is the Dedekind zeta function $\zeta_F(s)$). Let $\delta_{\pi}=1$ if $\pi=\mathbbm{1}$, and $\delta_{\pi}=0$ otherwise. If
\[
L_{\infty}(s,\pi)=\prod_{v\mid\infty}L(s,\pi_v),
\]
then the completed $L$-function
\[
\Lambda(s,\pi)=(s(1-s))^{\delta_{\pi}}(D_F^n\N\kq_{\pi})^{s/2}L_{\infty}(s,\pi)L(s,\pi)
\]
is entire of order $1$. There exists $W(\pi)\in\CC$ of modulus $1$ such that $\Lambda(s,\pi)$ satisfies the functional equation
\[
\Lambda(s,\pi)=W(\pi)\Lambda(1-s,\tilde\pi)=W(\pi)\overline{\Lambda(1-\bar{s},\pi)}.
\]
The nontrivial zeros of $L(s,\pi)$ are the zeros of $\Lambda(s,\pi)$, and they lie in the critical strip $0<\Re(s)<1$. It is conjectured (GRH) that they actually lie on the critical line $\Re(s)=1/2$. Finally, the analytic conductor is $C(\pi)=C(0,\pi)$, where
\begin{equation}
\label{eqn:AC_def_standard}
C(it,\pi) = D_F^n\N\kq_{\pi}\prod_{v\mid\infty}\prod_{j=1}^n (|\mu_{j,\pi}(v)+it|+3)^{[F_v:\R]}.
\end{equation}

More generally, these quantities are defined for all $\pi\in\mathfrak{F}_n$.  In particular, we have that
\[
L(s,\pi)=L(s+it_{\pi},\pi^*),\qquad C(it,\pi)=C(it+it_{\pi},\pi^*).
\]
We end this subsection by noting that $t_\pi$ can be expressed as a weighted average of the archimedean Langlands parameters of $\pi$. The proof is based on the Langlands classification of the admissible dual of $\GL_n$ over an archimedean local field, and the inductive nature of the corresponding local $L$-functions. We start with a general observation about central characters over an archimedean local field.

\begin{lemma}\label{lem:rhochar}
Let $K\in\{\R,\CC\}$, and let $\rho$ be a generic irreducible admissible representation of $\GL_n(K)$ with central character $\omega_\rho$. There exist complex numbers $\mu_1,\dotsc,\mu_n\in\CC$ such that
\begin{equation}\label{eq:Lrho}
L(s,\rho)=\prod_{j=1}^n\Gamma_K(s+\mu_j)
\end{equation}
and
\begin{equation}\label{eq:wrho}
\frac{\omega_\rho(a)}{|\omega_\rho(a)|}=\prod_{j=1}^n a^{i[K:\R]\Im(\mu_j)},\qquad a>0.
\end{equation}
\end{lemma}

\begin{proof}
First, we discuss the case of $K=\R$. By \cite[Theorem~1]{Knapp}, there exist a partition $n=n_1+\cdots+n_r$ with terms $n_k\in\{1,2\}$, and for each $k\in\{1,\dotsc,r\}$, a representation $\sigma_k$ of $\GL_{n_k}(\R)$ such that $\rho$ is infinitesimally equivalent to the unique irreducible quotient of the representation unitarily induced from $\sigma_1\otimes\cdots\otimes\sigma_r$. If $n_k=1$, then there exists $(\ell_k,t_k)\in\{0,1\}\times\CC$ such that $\sigma_k=\sgn^{\ell_k}\otimes|\cdot|^{t_k}$. The corresponding $L$-function is given by \cite[(3.4)~\&~(3.6)]{Knapp} as \[L(s,\sigma_k)=\Gamma_\R(s+t_k+\ell_k).\]
Note that if $a>0$, then $\sigma_k(a)$ acts by the scalar $a^{t_k}$. If $n_k=2$, then there exists $(\ell_k,t_k)\in\Z_{\geq 1}\times\CC$ such that $\sigma_k=D_{\ell_k}\otimes|\det|^{t_k}$. The corresponding $L$-function is given by \cite[(3.4)~\&~(3.6)]{Knapp} as
\[L(s,\sigma_k)=\Gamma_\CC(s+t_k+\ell_k/2)=\Gamma_\R(s+t_k+\ell_k/2)\Gamma_\R(s+t_k+\ell_k/2+1).\]
Note that if $a>0$, then $\sigma_k(\diag(a,a))$ acts by the scalar $a^{2t_k}$. The $L$-function of $\rho$ is the product of all these gamma factors:
\begin{equation}\label{eq:Lrho_real}
L(s,\rho)=\prod_{\substack{1\leq k\leq r\\n_k=1}}\Gamma_\R(s+t_k+\ell_k)
\prod_{\substack{1\leq k\leq r\\n_k=2}}\Gamma_\R(s+t_k+\ell_k/2)\Gamma_\R(s+t_k+\ell_k/2+1).
\end{equation}
Moreover, if $a>0$, then $\rho(\diag(a,\dotsc,a))$ acts by the scalar $\prod_{k=1}^r a^{n_k t_k}$. Therefore, if we write \eqref{eq:Lrho_real} in the compact form \eqref{eq:Lrho}, then it follows from the fact that each $\ell_k$ in \eqref{eq:Lrho_real} is real that
\[\frac{\omega_\rho(a)}{|\omega_\rho(a)|}=\prod_{k=1}^r a^{i\Im(n_kt_k)}=\prod_{j=1}^n a^{i\Im(\mu_j)},\qquad a>0.\]
This confirms \eqref{eq:wrho} when $K=\R$.

Now, we discuss the case of $K=\CC$. By \cite[Theorem~4]{Knapp}, for each $k\in\{1,\dotsc,n\}$, there exists $(\ell_k,t_k)\in\Z\times\CC$ such that if $\sigma_k$ is the quasi-character $z\mapsto(z/|z|)^{\ell_k} |z|^{2t_k}$ of $\CC^{\times}$, then $\rho$ is infinitesimally equivalent to the unique irreducible quotient of the representation unitarily induced from $\sigma_1\otimes\dots\otimes\sigma_n$. The corresponding $L$-function is given by \cite[(4.6)]{Knapp} as
\[L(s,\sigma_k)=\Gamma_\CC(s+t_k+|\ell_k|/2).\]
Note that if $a>0$, then $\sigma_k(a)$ acts by the scalar $a^{2t_k}$. The $L$-function of $\rho$ is the product of all these gamma factors:
\begin{equation}\label{eq:Lrho_complex}
L(s,\rho)=\prod_{k=1}^n\Gamma_\CC(s+t_k+|\ell_k|/2).
\end{equation}
Moreover, if $a>0$, then $\rho(\diag(a,\dotsc,a))$ acts by the scalar $\prod_{k=1}^n a^{2t_k}$. Therefore, if we write \eqref{eq:Lrho_complex} in the compact form \eqref{eq:Lrho}, then it follows from the fact that each $|\ell_k|$ in \eqref{eq:Lrho_complex} is real that
\[\frac{\omega_\rho(a)}{|\omega_\rho(a)|}=\prod_{k=1}^n a^{i\Im(2t_k)}=\prod_{j=1}^n a^{i\Im(2\mu_j)},\qquad a>0.\]
This confirms \eqref{eq:wrho} when $K=\CC$.
\end{proof}

\begin{lemma}
\label{tpiformula}
Let $\pi\in\mathfrak{F}_n$ with archimedean Langlands parameters $\mu_{j,\pi}(v)$ as in \eqref{eq:Lpiv}.  If
\begin{equation}\label{eq:mpi}
m_{\pi}(v)=[F_v:\R]\sum_{j=1}^n \Im(\mu_{j,\pi}(v)),
\end{equation}
then
\begin{equation}\label{eq:tpi}
n[F:\Q]t_\pi=\sum_{v\mid\infty}m_{\pi}(v).
\end{equation}
In particular,
\begin{equation}\label{eq:tpibound}
|t_\pi|\leq\max_{\substack{v\mid\infty\\1\leq j\leq n}}|\Im(\mu_{j,\pi}(v)|.
\end{equation}
\end{lemma}

\begin{proof}
Since the central character $\omega_\pi$ is unitary, it follows from \eqref{eq:Lpiv}, \eqref{eq:mpi}, and \cref{lem:rhochar} that
\[
\omega_{\pi_v}(a)=a^{im_{\pi}(v)},\qquad v\mid\infty,\quad a>0.
\]
Now consider the archimedean central elements
\[
D(a)=\prod_{v\mid\infty}\diag(a,\dotsc,a)\in\GL_n(\A),\qquad a>0.
\]
Observe that $\det(D(a))$ has idele norm $a^{n[F:\Q]}$, and recall the definition of $t_\pi$ implicit in \eqref{eqn:pidecomp}. Calculating the action of $\pi(D(a))$ in two ways, we see that
\[
(a^{n[F:\Q]})^{it_\pi}=\omega_\pi(a)=\prod_{v\mid\infty}\omega_{\pi_v}(a)=\prod_{v\mid\infty}a^{im_{\pi}(v)},\qquad a>0.
\]
Comparing the two sides, the identity \eqref{eq:tpi} and the bound \eqref{eq:tpibound} follow readily.
\end{proof}

\subsection{Rankin--Selberg $L$-functions}

Let $\pi\in\mathfrak{F}_n^*$ and $\pi'\in\mathfrak{F}_{n'}^*$. For each $\kp\mid\kq_{\pi}\kq_{\pi'}$, there exist complex numbers $\alpha_{j,j',\pi\times\pi'}(\kp)$ with $1\leq j\leq n$ and $1\leq j'\leq n'$ such that if
\[
L(s,\pi_{\kp}\times\pi_{\kp}')=\begin{cases}
\prod_{j=1}^n \prod_{j'=1}^{n'}(1-\alpha_{j,\pi}(\kp)\alpha_{j',\pi'}(\kp)\N\kp^{-s})^{-1}&\mbox{if $\kp\nmid\kq_{\pi}\kq_{\pi'}$,}\\
\prod_{j=1}^n \prod_{j'=1}^{n'}(1-\alpha_{j,j',\pi\times\pi'}(\kp)\N\kp^{-s})^{-1}&\mbox{if $\kp\mid\kq_{\pi}\kq_{\pi'}$,}
\end{cases}
\]
then the Rankin--Selberg $L$-function $L(s,\pi\times\pi')$ equals the absolutely convergent product
\begin{equation}
\label{eqn:euler_prod}
L(s,\pi\times\pi')=\prod_{\kp}L(s,\pi_{\kp}\times\pi_{\kp}')=\sum_{\ka}\frac{\lambda_{\pi\times\pi'}(\ka)}{\N\ka^s},\qquad \Re(s)>1.	
\end{equation}

Let $\kq_{\pi\times\pi'}$ be the conductor of $L(s,\pi\times\pi')$. For each $v\mid\infty$, $1\leq j\leq n$, and $1\leq j'\leq n'$, there exists a Langlands parameter $\mu_{j,j',\pi\times\pi'}(v)$ such that if
\[
L_{\infty}(s,\pi\times\pi')=\prod_{v\mid\infty}\prod_{j=1}^n \prod_{j'=1}^{n'}\Gamma_v(s+\mu_{j,j',\pi\times\pi'}(v)),\quad \delta_{\pi\times\pi'}=\begin{cases}
	1&\mbox{if $\pi'=\tilde{\pi}$,}\\
	0&\mbox{otherwise,}
\end{cases}
\]
then the completed $L$-function
\[
\Lambda(s,\pi\times\pi')=(s(1-s))^{\delta_{\pi\times\pi'}}(D_F^{nn'}\N\kq_{\pi\times\pi'})^{s/2} L_{\infty}(s,\pi\times\pi')L(s,\pi\times\pi')
\]
is entire of order 1.  There exists $W(\pi\times\pi')\in\CC$ of modulus $1$ such that $\Lambda(s,\pi\times\pi')$ satisfies the functional equation
\[
\Lambda(s,\pi\times\pi')=W(\pi\times\pi')\Lambda(1-s,\tilde{\pi}\times\tilde{\pi}')=W(\pi\times\pi')\overline{\Lambda(1-\bar{s},\pi\times\pi')}.
\]

The absolute convergence of \eqref{eqn:euler_prod} ensures that $\re(\mu_{j,j',\pi\times\pi'}(v))\geq -1$. The nontrivial zeros of $L(s,\pi\times\pi')$ are the zeros of $\Lambda(s,\pi\times\pi')$, and they lie in the critical strip $0<\Re(s)<1$. It is conjectured (GRH) that they actually lie on the critical line $\Re(s)=1/2$. Finally, the analytic conductor is $C(\pi\times\pi')=C(0,\pi\times\pi')$, where
\[
C(it,\pi\times\pi')=D_F^{nn'}\N\kq_{\pi\times\pi'}\prod_{v\mid\infty}\prod_{j=1}^n \prod_{j'=1}^{n'}
(|\mu_{j,j',\pi\times\pi'}(v)+it|+3)^{[F_v:\R]}.
\]
We have bounds \cite[Lemma~A.1]{Wattanawanichkul}
\begin{equation}
\label{eqn:BH}
C(it,\pi\times\pi')\leq C(\pi\times\pi') (|t|+3)^{nn'[F:\Q]}\leq C(\pi)^{n'}C(\pi')^n (|t|+3)^{nn'[F:\Q]}.
\end{equation}

More generally, for $\pi\in\mathfrak{F}_{n}$ and $\pi'\in\mathfrak{F}_{n'}$, we have
\[
L(s,\pi\times\pi') = L(s+it_{\pi}+it_{\pi'},\pi^*\times\pi'^*),\quad 
C(it,\pi\times\pi')=C(it+it_{\pi}+it_{\pi'},\pi^*\times\pi'^*).
\]
Furthermore, we have that
\begin{equation}
\label{eqn:conjugate_symmetry}
L(s,\tilde{\pi}\times\tilde{\pi}') = \overline{L(\bar{s},\pi\times\pi')}.
\end{equation}

\subsection{Convexity}

Our proofs require strong bounds for Rankin--Selberg $L$-functions and their derivatives. 

\begin{lemma}[{\cite[Lemma~3.2]{HarcosThorner}}]
\label{lem:Li1}
For $(\pi,\pi')\in\mathfrak{F}_n\times\mathfrak{F}_{n'}$, consider the holomorphic function
\[
\cL(s,\pi\times\pi')=\left(\frac{s+it_{\pi}+it_{\pi'}-1}{s+it_{\pi}+it_{\pi'}+1}\right)^{\delta_{\pi^*\times\pi'^*}}L(s,\pi\times\pi'),\qquad\Re(s)>-1.
\]
If $j\geq 0$, $\sigma\geq 0$, and $\epsilon>0$, then
\[
\cL^{(j)}(\sigma,\pi\times\pi')\ll_{n,n',[F:\Q],j,\epsilon}C(\pi\times\pi')^{\max(1-\sigma,0)/2+\epsilon}.
\]
\end{lemma}

\section{Goldfeld--Hoffstein--Lieman type arguments}

The main goal of this section is to prove \cref{lem:GHLnew} below.  This generalizes and refines the Lemma in \cite[Appendix]{HoffsteinLockhart}.

\subsection{Isobaric sums}

The Langlands theory of Eisenstein series associates to any $\ell$-tuple 
$(\pi_1,\dotsc,\pi_\ell)\in\mathfrak{F}_{n_1}\times\dotsb\times \mathfrak{F}_{n_\ell}$ an automorphic representation of $\GL_{n_1+\dotsb+n_\ell}(\A_F)$, the isobaric sum, denoted $\Pi=\pi_1\boxplus\dotsb\boxplus\pi_\ell$. The contragredient is $\tilde{\Pi}=\tilde{\pi}_1\boxplus\dotsb\boxplus\tilde{\pi}_\ell$, and
\[
L(s,\Pi)=\prod_{j=1}^\ell L(s,\pi_j),\qquad \Re(s)>1.
\]
We let $\mathfrak{A}_n$ denote the set of isobaric automorphic representations of $\GL_{n}(\A_F)$.  Given $\Pi=\pi_1\boxplus\dotsb\boxplus\pi_\ell\in\mathfrak{A}_n$ and $\Pi'=\pi_1'\boxplus\dotsb\boxplus\pi_m'\in\mathfrak{A}_{n'}$, we define the Rankin--Selberg $L$-function
\[
L(s,\Pi\times\Pi')=\prod_{j=1}^\ell \prod_{k=1}^m L(s,\pi_j\times\pi_k')=\sum_{\ka}\frac{\lambda_{\Pi\times\Pi'}(\ka)}{\N\ka^s},\qquad \Re(s)>1.
\]
Its analytic conductor is $C(\Pi\times\Pi')=C(0,\Pi\times\Pi')$, where
\[
C(it,\Pi\times\Pi') = \prod_{j=1}^n \prod_{k=1}^{n'}C(it,\pi_j\times\pi_{k}').
\]

\begin{lemma}[{\cite[Lemma~a]{HoffsteinRamakrishnan}}]
\label{lem:nonneg}
If $\Pi\in\mathfrak{A}_n$, then $L(s,\Pi\times\tilde{\Pi})$ has nonnegative Dirichlet coefficients, as do $\log L(s,\Pi\times\tilde{\Pi})$ and $-L'(s,\Pi\times\tilde{\Pi})/L(s,\Pi\times\tilde{\Pi})$.
\end{lemma}

The next result generalizes the Lemma in \cite[Appendix]{HoffsteinLockhart}.

\begin{lemma}
\label{lem:GHLnew}
Let $\Pi=\pi_1\boxplus\dotsb\boxplus \pi_\ell\in\mathfrak{A}_n$, and let $r$ denote the order of the pole of $L(s,\Pi\times\tilde{\Pi})$ at $s=1$. Let
\begin{equation}
\label{eq:Edef}
E(\Pi\times\tilde{\Pi})=\log C(\Pi\times\tilde{\Pi})+\sum_{\substack{1\leq j<k\leq\ell\\\pi_j^*=\pi_k^*\\\pi_j\neq\pi_k}}\frac{1}{|t_{\pi_j}-t_{\pi_k}|}.
\end{equation}
Then $L(\sigma+it,\Pi\times\tilde{\Pi})$ has at most $r$ zeros (counted with multiplicity) in the region
\[
\sigma\geq 1-\frac{1}{7rE(\Pi\times\tilde{\Pi})}
\qquad\text{and}\qquad
|t|\leq\frac{1}{12\sqrt{r}E(\Pi\times\tilde{\Pi})}.
\]
\end{lemma}

\begin{remark}
\label{rem:E}
The terms of the $(j,k)$-sum in \eqref{eq:Edef} are the reciprocal imaginary parts of the poles of $L(s,\Pi\times\tilde\Pi)$ in the upper half-plane, counted with multiplicity.
\end{remark}

\begin{proof}
Let $\sigma\in(1,2)$. We apply \cite[Lemma~3.1]{Wattanawanichkul} to each factor
\[L(\sigma,\pi_j\times\tilde\pi_k)=L(\sigma+it_{\pi_j}-it_{\pi_k},\pi_j^*\times\tilde\pi_k^*)\]
of $L(\sigma,\Pi\times\tilde{\Pi})$ to conclude that
\begin{align*}
\sum_{L(\rho,\Pi\times\tilde{\Pi})=0}\re\left(\frac{1}{\sigma-\rho}\right)\leq
&\sum_{\substack{1\leq j,k\leq\ell\\\pi_j^*=\pi_k^*}}\re\left(\frac{1}{\sigma-1+it_{\pi_j}-it_{\pi_k}}\right)\\
&+\frac{\ell^2}{2}+\frac{1}{2}\log C(\Pi\times\tilde{\Pi})+\frac{L'}{L}(\sigma,\Pi\times\tilde{\Pi}).
\end{align*}
The last term is nonpositive by \cref{lem:nonneg}, hence
\begin{align*}
\sum_{L(\rho,\Pi\times\tilde{\Pi})=0}\re\left(\frac{1}{\sigma-\rho}\right)
&\leq\sum_{\substack{1\leq j,k\leq\ell\\\pi_j^*=\pi_k^*}}\frac{\sigma-1}{(\sigma-1)^2+(t_{\pi_j}-t_{\pi_k})^2}
+\log C(\Pi\times\tilde{\Pi})\\
&\leq\frac{r}{\sigma-1}+\sum_{\substack{1\leq j<k\leq\ell\\\pi_j^*=\pi_k^*\\\pi_j\neq\pi_k}}\frac{1}{|t_{\pi_j}-t_{\pi_k}|}
+\log C(\Pi\times\tilde{\Pi}).
\end{align*}
Let
\[
\sigma = 1+\frac{1}{2E(\Pi\times\tilde{\Pi})},
\]
and let $\mathcal{N}\geq 0$ be the number of zeros $\rho=\beta+i\gamma$ (counted with multiplicity) of $L(s,\Pi\times\tilde{\Pi})$ satisfying
\[
\beta\geq 1-\frac{1}{7rE(\Pi\times\tilde{\Pi})}
\qquad\text{and}\qquad
|\gamma|\leq\frac{1}{12\sqrt{r}E(\Pi\times\tilde{\Pi})}.
\]
These inequalities imply that
\[
(\sigma-\beta)^2+\gamma^2\leq(\sigma-\beta)^2+\frac{(\sigma-1)^2}{36r}\leq\left(1+\frac{1}{36r}\right)(\sigma-\beta)^2,
\]
hence also that
\[
\re\left(\frac{1}{\sigma-\rho}\right)=\frac{\sigma-\beta}{(\sigma-\beta)^2+\gamma^2}\geq
\left(1+\frac{1}{36r}\right)^{-1}(\sigma-\beta)^{-1}.
\]
We conclude via nonnegativity that
\[
\frac{\mathcal{N}}{\left(1+\frac{1}{36r}\right)\left(\sigma-1+\frac{1}{7rE(\Pi\times\tilde{\Pi})}\right)}\leq\sum_{L(\rho,\Pi\times\tilde{\Pi})=0}\re\left(\frac{1}{\sigma-\rho}\right)\leq \frac{r}{\sigma-1}+E(\Pi\times\tilde{\Pi}).
\]
A comparison of the two sides shows that
\[
\mathcal{N}\leq\left(1+\frac{1}{36r}\right)\left(\frac{1}{2}+\frac{1}{7r}\right)(2r+1)<r+1.
\]
Since $\mathcal{N}$ is an integer, it follows that $\mathcal{N}\leq r$, as desired.
\end{proof}

\subsection{Applications of \cref{lem:GHLnew}}
For $(\pi,\pi')\in\mathfrak{F}_n\times\mathfrak{F}_{n'}$ and $(\chi,\chi_1,\chi_2)\in\mathfrak{F}_1^3$, we define
\begin{alignat}{3}
\label{eqn:Pi}
\Pi&=\pi\boxplus\pi\otimes\chi\boxplus\tilde\pi'\boxplus\tilde\pi'\otimes\bar{\chi},\qquad& 
D(s)&=D(s;\pi,\pi',\chi)=L(s,\Pi\times\tilde{\Pi}),\\
\label{eqn:Pi2}
\Pi'&=\pi\boxplus\tilde\pi'\otimes\bar{\chi},& 
H(s)&=H(s;\pi,\pi',\chi)=L(s,\Pi'\times\tilde{\Pi}'),\\
\label{eqn:Pi3}
\Pi''&=\pi\boxplus \pi\otimes\chi_1\boxplus\pi\otimes\chi_2,& 
I(s)&=I(s;\pi,\chi_1,\chi_2)=L(s,\Pi''\times\tilde{\Pi}'').
\end{alignat}
It follows from \cref{lem:nonneg} that the above $L$-functions have nonnegative Dirichlet coefficients. Moreover, they factor into simpler $L$-functions as follows:
\begin{align}
D(s)=
\label{eqn:D_def}
&~L(s,\pi\times\tilde\pi)^2 L(s,\pi'\times\tilde\pi')^2 L(s,\pi\times(\pi'\otimes\chi))^2 L(s,\tilde\pi\times(\tilde\pi'\otimes\bar{\chi}))^2\\
\notag
&\cdot L(s,\pi\times(\tilde\pi\otimes\chi)) L(s,\pi'\times(\tilde\pi'\otimes\chi)) L(s,\tilde\pi\times\tilde\pi') L(s,\pi\times(\pi'\otimes\chi^2))\\
\notag
&\cdot L(s,\pi\times(\tilde\pi\otimes\bar{\chi})) L(s,\pi'\times(\tilde\pi'\otimes\bar{\chi})) L(s,\pi\times\pi') L(s,\tilde\pi\times(\tilde\pi'\otimes\bar{\chi}^2)),\\[6pt]
\label{eqn:D_def2}
H(s)=&~L(s,\pi\times\tilde\pi)L(s,\pi'\times\tilde\pi')L(s,\pi\times(\pi'\otimes\chi))L(s,\tilde\pi\times(\tilde\pi'\otimes\bar{\chi})),\\[6pt]
\label{eqn:D_def3}
I(s)=&~L(s,\pi\times\tilde{\pi})^3 L(s,\pi\times(\tilde{\pi}\otimes\chi_1))L(s,\pi\times(\tilde{\pi}\otimes\bar{\chi}_1))L(s,\pi\times(\tilde{\pi}\otimes\chi_2))\\
\notag
&\cdot L(s,\pi\times(\tilde{\pi}\otimes\bar{\chi}_2))L(s,\pi\times(\tilde{\pi}\otimes\chi_1\bar{\chi}_2))L(s,\pi\times(\tilde{\pi}\otimes\bar{\chi}_1 \chi_2)).
\end{align}
Using \eqref{eqn:BH}, we conveniently bound the analytic conductor of $D(s)$ (resp. $H(s)$) by $Q^4$ (resp. $Q$), where
\begin{equation}
\label{eqn:Rdef}
Q=Q(\pi,\pi',\chi)=(C(\pi)C(\pi'))^{2(n+n')}C(\chi)^{(n+n')^2}.
\end{equation}
Similarly, we conveniently bound the analytic conductor of $I(s)$ by $R^9$, where
\begin{equation}
\label{eqn:Rdef2}
R=R(\pi,\chi_1,\chi_2)=C(\pi)^{2n}(C(\chi_1)C(\chi_2))^{n^2/2}.
\end{equation}

\begin{lemma}
\label{lem:GHL}
Let $(\pi,\pi',\chi)\in\mathfrak{F}_n\times\mathfrak{F}_{n'}\times\mathfrak{F}_1$ and $\epsilon\in(0,1)$.  Let $D(s)=D(s;\pi,\pi',\chi)$ be as in \eqref{eqn:D_def}, and $Q=Q(\pi,\pi',\chi)$ as in \eqref{eqn:Rdef}.  Assume that the $L$-functions
\begin{equation}
\label{eqn:Lfns}
L(s,\pi\times\pi'),\qquad L(s,\pi\times(\pi'\otimes\chi)),\qquad L(s,\pi\times(\pi'\otimes\chi^2))
\end{equation}
are entire.  If $\pi\otimes\chi^*=\pi$ or $\pi'\otimes\chi^*=\pi'$, then assume also that $|t_{\chi}|\geq Q^{-\epsilon}$.  Then $D(\sigma+it)$ has at most four zeros (counted with multiplicity) in the region
\begin{equation}
\label{eqn:GHL_interval}
\sigma\geq 1-\epsilon 2^{-8}Q^{-\epsilon}
\qquad\text{and}\qquad
|t|\leq\epsilon 2^{-8}Q^{-\epsilon}.
\end{equation}
\end{lemma}
\begin{proof}
We apply \cref{lem:GHLnew} to the situation described in \eqref{eqn:Pi} and \eqref{eqn:D_def}. By the initial assumptions, $D(s)$ has a quadruple pole at $s=1$, potential simple or double poles at $s=1\pm it_\chi$, and no other pole. Hence $r=4$ in \cref{lem:GHLnew}, while by $C(\Pi\times\tilde{\Pi})\leq Q^4$ and \eqref{eq:Edef} (see also \cref{rem:E}),
\[
E(\Pi\times\tilde{\Pi})\leq 4\log Q+2Q^\epsilon<(6/\epsilon)Q^{\epsilon}.
\]
The desired result follows (with room to spare).
\end{proof}

\begin{lemma}
\label{lem:GHL2}
Let $(\pi,\pi',\chi)\in\mathfrak{F}_n\times\mathfrak{F}_{n'}\times\mathfrak{F}_1$. Let $H(s)=H(s;\pi,\pi',\chi)$ be as in \eqref{eqn:D_def2}, and assume that $L(s,\pi\times(\pi'\otimes\chi))$ is entire. Put $\Cl[abcon]{GHLnew2}=1/(34(n+n'))$. Then $H(\sigma+it)$ has at most two zeros (counted with multiplicity) in the region
\[
\sigma\geq 1-\frac{\Cr{GHLnew2}}{\log(C(\pi)C(\pi'\otimes\chi))}
\qquad\text{and}\qquad
|t|\leq\frac{\Cr{GHLnew2}}{\log(C(\pi)C(\pi'\otimes\chi))}.
\]
\end{lemma}
\begin{proof}
We apply \cref{lem:GHLnew} to the situation described in \eqref{eqn:Pi2} and \eqref{eqn:D_def2}. By the initial assumptions, $H(s)$ has a double pole at $s=1$ and no other pole. Hence $r=2$ in \cref{lem:GHLnew}, while by \eqref{eq:Edef} (see also \cref{rem:E}) and \eqref{eqn:BH},
\[
E(\Pi'\times\tilde{\Pi}')=\log C(\Pi'\times\tilde{\Pi}')\leq 2(n+n')\log(C(\pi)C(\pi'\otimes\chi)).
\]
The desired result follows.
\end{proof}

We now use \cref{lem:GHLnew} to prove \cref{thm:standard_region}.

\begin{proof}[Proof of \cref{thm:standard_region}]
We begin by establishing the zero-free region \eqref{eqn:ZFR_standard_region}.  We distinguish between three cases depending on whether $\pi\otimes\chi^2=\pi$ or even $\pi\otimes\chi=\pi$.

\emph{Case 1:} $\pi\otimes\chi^2\neq\pi$ (which automatically implies $\pi\otimes\chi\neq\pi$).  Since $\chi\in\mathfrak{F}_1^*$, the $L$-functions
\[
L(s,\pi\times(\tilde{\pi}\otimes\chi)),\qquad L(s,\pi\times(\tilde{\pi}\otimes\chi^2))
\]
are entire. We write $s=\sigma+it$ and set out to prove that $L(s,\pi\times(\tilde\pi\otimes\chi))$ has no zero in the region \eqref{eqn:ZFR_standard_region}. We let $t\in\R$ and apply \cref{lem:GHLnew} to 
\begin{equation}
\label{specialPi}
\Pi''=\pi\boxplus\pi\otimes\chi|\cdot|^{it}\boxplus\pi\otimes\bar{\chi}|\cdot|^{-it}.
\end{equation}
By \eqref{eqn:Pi3} and \eqref{eqn:D_def3}, the corresponding auxiliary $L$-function equals
\begin{align}
\label{specialD}
L(s,\Pi''\times\tilde{\Pi}'')=
&~L(s,\pi\times\tilde\pi)^3 L(s+it,\pi\times(\tilde\pi\otimes\chi))^2 L(s-it,\tilde\pi\times(\pi\otimes\bar{\chi}))^2\\
\notag
&\cdot L(s+2it,\pi\times(\tilde\pi\otimes\chi^2)) L(s-2it,\tilde\pi\times(\pi\otimes\bar{\chi}^2)),
\end{align}
which has a triple pole at $s=1$ and no other pole. Hence $r=3$ in \cref{lem:GHLnew}, while by 
\eqref{eqn:BH}, \eqref{eq:Edef} (see also \cref{rem:E}), and \eqref{eqn:Rdef2},
\begin{align*}
E(\Pi''\times\tilde{\Pi}'')
&=\log C(\Pi''\times\tilde{\Pi}'')\\
&\leq 9\log R(\pi,\chi|\cdot|^{it},\bar{\chi}|\cdot|^{-it})\\
&\leq 9\log(C(\pi)^{2n}C(\chi)^{n^2}(|t|+3)^{n^2[F:\Q]}).
\end{align*}
By \eqref{eqn:conjugate_symmetry}, if $\sigma\in\R$ satisfies $L(\sigma+it,\pi\times(\tilde{\pi}\otimes\chi))=0$, then also $L(\sigma-it,\pi\times(\tilde{\pi}\otimes\bar{\chi}))=0$.  Therefore, $\sigma$ is a real zero of $L(s,\Pi''\times\tilde{\Pi}'')$ with multiplicity at least $4$.  This establishes the zero-free region \eqref{eqn:ZFR_standard_region} when $\pi\otimes\chi^2\neq\pi$, lest we contradict \cref{lem:GHLnew}.

\emph{Case 2:} $\pi\otimes\chi^2=\pi$ and $\pi\otimes\chi\neq\pi$. In this case, $\pi\otimes\chi=\pi\otimes\bar{\chi}$ differs from $\pi$, hence
$L(s,\pi\times(\tilde{\pi}\otimes\chi))$ is self-dual by \eqref{eqn:conjugate_symmetry} and entire.  We apply \cref{lem:GHL2} with $\pi'=\tilde\pi$. By \eqref{eqn:D_def2} and \cref{lem:GHL2}, the auxiliary $L$-function
\[H(s;\pi,\tilde\pi,\chi)=L(s,\pi\times\tilde\pi)^2L(s,\pi\times(\tilde{\pi}\otimes\chi))^2\]
has at most two zeros (counted with multiplicity) in the region
\begin{equation}
\label{eqn:low_region}
\sigma\geq 1-\frac{1}{68\log(C(\pi)^{2n}C(\chi)^{n^2})}
\qquad\text{and}\qquad
|t|\leq\frac{1}{68\log(C(\pi)^{2n}C(\chi)^{n^2})}.
\end{equation}
Since $L(s,\pi\times(\tilde{\pi}\otimes\chi))$ is a self-dual $L$-function, its non-real zeros come in complex conjugate pairs per \eqref{eqn:conjugate_symmetry}.  It follows that a non-real zero of $L(s,\pi\times(\tilde{\pi}\otimes\chi))$ in \eqref{eqn:low_region} gives rise to two distinct zeros of $H(s;\pi,\tilde\pi,\chi)$ in \eqref{eqn:low_region}, each with multiplicity at least two. Therefore, $L(s,\pi\times(\tilde{\pi}\otimes\chi))$ can only have real zeros in \eqref{eqn:low_region}. Similarly, a real zero of $L(s,\pi\times(\tilde{\pi}\otimes\chi))$ with multiplicity at least two cannot exist in \eqref{eqn:low_region}.  We conclude that $L(s,\pi\times(\tilde{\pi}\otimes\chi))$ 
has at most one zero (necessarily real and simple) in \eqref{eqn:low_region}. It remains to show that $L(s,\pi\times(\tilde{\pi}\otimes\chi))$ has no zero in the region \eqref{eqn:ZFR_standard_region} when
\begin{equation}
\label{tlowerbound}
|t|>\frac{1}{68\log(C(\pi)^{2n}C(\chi)^{n^2})}.
\end{equation}
Given such a $t\in\R$, we go back to our earlier construction \eqref{specialPi}. Then \eqref{specialD} simplifies to
\[
L(s,\Pi''\times\tilde{\Pi}'')=
L(s,\pi\times\tilde\pi)^3\prod_{\pm}L(s\pm it,\pi\times(\tilde\pi\otimes\chi))^2\prod_{\pm} L(s\pm 2it,\pi\times\tilde\pi).
\]
This auxiliary $L$-function has a triple pole at $s=1$, simple poles at $s=1\pm 2it$, and no other pole.  Hence $r=3$ in \cref{lem:GHLnew}, while by \eqref{eqn:BH}, \eqref{eq:Edef} (see also \cref{rem:E}), \eqref{eqn:Rdef2}, and \eqref{tlowerbound},
\begin{align*}
E(\Pi''\times\tilde{\Pi}'')
&=\log C(\Pi''\times\tilde{\Pi}'')+\frac{1}{2|t|}\\
&\leq 9\log R(\pi,\chi|\cdot|^{it},\bar{\chi}|\cdot|^{-it})+34\log(C(\pi)^{2n}C(\chi)^{n^2})\\
&\leq 43\log(C(\pi)^{2n}C(\chi)^{n^2}(|t|+3)^{n^2[F:\Q]}).
\end{align*}
Now, we finish as in Case 1. If $L(\sigma+it,\pi\times(\tilde{\pi}\otimes\chi))=0$, then $\sigma$ is a real zero of $L(s,\Pi''\times\tilde{\Pi}'')$ with multiplicity at least $4$.  This establishes the zero-free region \eqref{eqn:ZFR_standard_region} when $\pi\otimes\chi^2=\pi$ and $\pi\otimes\chi\neq\pi$.

\emph{Case 3:} $\pi\otimes\chi=\pi$ (which automatically implies $\pi\otimes\chi^2=\pi$). We apply \cref{lem:GHLnew} for $\Pi=\pi$, and then we also apply \eqref{eqn:BH}, to see that $L(s,\pi\times\tilde\pi)$ has at most one zero (counted with multiplicity) in the region
\[
\sigma\geq 1-\frac{1}{7\log(C(\pi)^{2n})}
\qquad\text{and}\qquad
|t|\leq\frac{1}{12\log(C(\pi)^{2n})}.
\]
Since $L(s,\pi\times\tilde{\pi})$ is a self-dual $L$-function, its non-real zeros come in complex conjugate pairs per \eqref{eqn:conjugate_symmetry}.  So the potential single zero of $L(s,\pi\times\tilde{\pi})$ in the region above must be real. It remains to show that $L(s,\pi\times\tilde{\pi})$ has no zero in the region \eqref{eqn:ZFR_standard_region} when
\begin{equation}
\label{tlowerbound2}
|t|>\frac{1}{12\log(C(\pi)^{2n})}.
\end{equation}
Given such a $t\in\R$, we go back to our earlier construction \eqref{specialPi}. Then \eqref{specialD} simplifies to
\[
L(s,\Pi''\times\tilde{\Pi}'')=
L(s,\pi\times\tilde\pi)^3\prod_{\pm}L(s\pm it,\pi\times\tilde\pi)^2\prod_{\pm} L(s\pm 2it,\pi\times\tilde\pi).
\]
This auxiliary $L$-function has a triple pole at $s=1$, double poles at $s=1\pm it$, simple poles at $s=1\pm 2it$, and no other pole.  Hence $r=3$ in \cref{lem:GHLnew}, while by \eqref{eqn:BH}, \eqref{eq:Edef} (see also \cref{rem:E}), \eqref{eqn:Rdef2}, and \eqref{tlowerbound2},
\begin{align*}
E(\Pi''\times\tilde{\Pi}'')
&=\log C(\Pi''\times\tilde{\Pi}'')+\frac{5}{2|t|}\\
&\leq 9\log R(\pi,|\cdot|^{it},|\cdot|^{-it})+30\log(C(\pi)^{2n})\\
&\leq 39\log(C(\pi)^{2n}(|t|+3)^{n^2[F:\Q]}).
\end{align*}
Now, we finish as in Case 1. If $L(\sigma+it,\pi\times\tilde{\pi})=0$, then $\sigma$ is a real zero of $L(s,\Pi''\times\tilde{\Pi}'')$ with multiplicity at least $4$.  This establishes the zero-free region \eqref{eqn:ZFR_standard_region} when $\pi\otimes\chi=\pi$.

We have established the zero-free region \eqref{eqn:ZFR_standard_region}, and we proceed to prove the rest of \cref{thm:standard_region}.  Let $(\pi,\chi_1,\chi_2)\in\mathfrak{F}_n\times\mathfrak{F}_1^*\times\mathfrak{F}_1^*$ satisfy $\pi\otimes\chi_1\neq\pi\otimes\chi_2$, and assume for the sake of contradiction that $L(s,\pi\times(\tilde{\pi}\otimes\chi_1))L(s,\pi\times(\tilde{\pi}\otimes\chi_2))$ has at least two zeros (counted with multiplicity) in the interval \eqref{twobetasinterval}.  We apply \cref{lem:GHLnew} to the situation described in \eqref{eqn:Pi3} and \eqref{eqn:D_def3}.

If neither $\pi\otimes\chi_1$ nor $\pi\otimes\chi_2$ equals $\pi$, then by $\pi\otimes\chi_1\neq\pi\otimes\chi_2$ and \eqref{eqn:D_def3}, the only singularity of $L(s,\Pi''\times\tilde{\Pi}'')$ is a triple pole at $s=1$.  Therefore, we have $r=3$ in \cref{lem:GHLnew}.  Furthermore, by \eqref{eq:Edef} (see also \cref{rem:E}) and \eqref{eqn:Rdef2},
\begin{equation}
\label{Ebound}
E(\Pi''\times\tilde{\Pi}'')=\log C(\Pi''\times\tilde{\Pi}'')\leq 9\log(C(\pi)^{2n}C(\chi_1)^{n^2/2}C(\chi_2)^{n^2/2}).
\end{equation}
By \eqref{eqn:conjugate_symmetry} and our indirect assumption, $L(s,\Pi''\times\tilde{\Pi}'')$ has at least four zeros (counted with multiplicity) in the interval \eqref{twobetasinterval}.  This contradicts \cref{lem:GHLnew} (with room to spare).

If $\pi\otimes\chi_2=\pi$ (say), then we argue similarly.  In this case, \eqref{eqn:D_def3} simplifies to
\[
L(s,\Pi''\times\tilde{\Pi}'')=L(s,\pi\times\tilde{\pi})^5 L(s,\pi\times(\tilde{\pi}\otimes\chi_1))^2 L(s,\pi\times(\tilde{\pi}\otimes\bar{\chi}_1))^2.
\]
We now have $r=5$ in \cref{lem:GHLnew}, while the bound \eqref{Ebound} is still valid. By \eqref{eqn:conjugate_symmetry} and our indirect assumption, $L(s,\Pi''\times\tilde{\Pi}'')$ has at least eight zeros (counted with multiplicity) in the interval \eqref{twobetasinterval}. This again contradicts \cref{lem:GHLnew} (with room to spare).
\end{proof}

\section{Key propositions}

Our proof of \cref{thm:Tatuzawa} rests on a more uniform version of \cite[Proposition~4.1]{HarcosThorner}.

\begin{proposition}
\label{prop:P1}
Let $(\pi,\pi',\chi)\in\mathfrak{F}_n\times\mathfrak{F}_{n'}\times\mathfrak{F}_1$ and $\epsilon\in(0,1)$.  Recall the notation in \eqref{eqn:pidecomp} and \eqref{eqn:Rdef}.  Assume that
\begin{itemize}
	\item $L(s,\pi\times\pi')$ is entire and has a zero in the interval $[1-\epsilon/16,1)$,
	\item $L(s,\pi\times(\pi'\otimes\chi))$ is entire, and
	\item if $\pi\otimes\chi^*=\pi$ or $\pi'\otimes\chi^*=\pi'$, then $|t_{\chi}|\geq Q^{-\epsilon/64}$.
\end{itemize}
There exists an effectively computable constant $\Cl[abcon]{P1}=\Cr{P1}(n,n',[F:\Q],\epsilon)>0$ such that
\begin{equation}
\label{ZFR_beta}
L(\sigma,\pi\times(\pi'\otimes\chi))\neq 0,\qquad \sigma\geq 1-\Cr{P1}Q^{-\epsilon}.
\end{equation}
\end{proposition}

We shall also use the following simpler variant of \cref{prop:P1}.

\begin{proposition}
\label{prop:P2}
Let $(\pi,\pi',\chi)\in\mathfrak{F}_n\times\mathfrak{F}_{n'}\times\mathfrak{F}_1$ and $\epsilon\in(0,1)$.  Recall the notation in \eqref{eqn:Rdef}.  Assume that
\begin{itemize}
	\item $L(s,\pi\times\tilde\pi)$ has a zero in the interval $[1-\epsilon/16,1)$, and
	\item $L(s,\pi\times(\pi'\otimes\chi))$ is entire.
\end{itemize}
There exists an effectively computable constant $\Cl[abcon]{P2}=\Cr{P2}(n,n',[F:\Q],\epsilon)>0$ such that
\begin{equation}
\label{ZFR_beta2}
L(\sigma,\pi\times(\pi'\otimes\chi))\neq 0,\qquad \sigma\geq 1-\Cr{P2}Q^{-\epsilon}.
\end{equation}
\end{proposition}

We will prove Propositions~\ref{prop:P1} and \ref{prop:P2} at the end of this section. These proofs rely on a residue calculation facilitated by the following lemma borrowed from \cite{HarcosThorner}.

\begin{lemma}[{\cite[Lemma~5.2]{HarcosThorner}}]
\label{lem:residue}
Let $f_0(s),\dotsc,f_m(s)$ be $m+1$ complex functions that are holomorphic in an open neighborhood of $s_0\in\CC$. If there exists $c\geq 0$ such that $|f_j(s_0)|=c$ for all $j\in\{1,\dotsc,m\}$, then
\[
\mathop{\mathrm{Res}}_{s=s_0}\frac{f_0(s)\dotsb f_m(s)}{(s-s_0)^m}
\]
equals $c$ times a $\CC$-linear combination of monomials of the derivative values $f_j^{(k)}(s_0)$ for $(j,k)\in\{0,\dotsc,m\}\times\{0,\dotsc,m-1\}$. The monomials in the linear combination, and the modulus of each coefficient in the linear combination depend at most on $m$.
\end{lemma}

We apply \cref{lem:residue} to study the residues of $D(s)$ and $H(s)$.  The information we need about the residues of $D(s)$ is as follows.

\begin{lemma}
\label{lem:residue_bounds}
Let $(\pi,\pi',\chi)\in\mathfrak{F}_n\times\mathfrak{F}_{n'}\times\mathfrak{F}_1$, $x>1$, $\epsilon\in(0,1)$, and $\beta\in[1-\epsilon/2,1)$. Recall the notation in  \eqref{eqn:pidecomp}, \eqref{eqn:D_def}, and \eqref{eqn:Rdef}.  Let $\mathcal{S}$ be the set of poles of $D(s)=D(s;\pi,\pi',\chi)$. Assume that the $L$-functions in \eqref{eqn:Lfns} are entire.  If $\pi\otimes\chi^*=\pi$ or $\pi'\otimes\chi^*=\pi'$, then assume also that $|t_\chi|\geq Q^{-\epsilon/8}$.  There holds
\[
\sum_{s_0\in\mathcal{S}}\mathop{\mathrm{Res}}_{s=s_0}D(s)x^{s-\beta}\Gamma(s-\beta)
\ll_{n,n',[F:\Q],\epsilon}(1-\beta)^{-4}|L(1,\pi\times(\pi'\otimes\chi))|(Qx)^{\epsilon}.	
\]
\end{lemma}

\begin{proof}
First, assume that $\pi\otimes\chi^*=\pi$ and $\pi'\otimes\chi^*=\pi'$. Then $|t_{\chi}|\geq Q^{-\epsilon/8}$ by hypothesis, and $\mathcal{S}=\{1,1-it_\chi,1+it_\chi\}$. For each choice of $s_0\in\mathcal{S}$, let $m$ be the order of the pole of $D(s)$ at $s=s_0$. Specifically, $m=4$ for $s_0=1$, and $m=2$ for $s_0=1\pm it_\chi$. Consider the decomposition
\begin{equation}
\label{eqn:decomposition}
(s-s_0)^m D(s)x^{s-\beta}\Gamma(s-\beta)=f_0(s)\dotsb f_m(s),
\end{equation}
where
\begin{itemize}
\item $f_1(s)=f_2(s)=L(s+it_\chi,\pi\times\pi')$ and $f_3(s)=f_4(s)=L(s-it_\chi,\tilde{\pi}\times\tilde{\pi}')$ for $s_0=1$;
\item $f_1(s)=L(s,\tilde\pi\times\tilde\pi')$ and $f_2(s)=L(s+2it_\chi,\pi\times\pi')$ for $s_0=1-it_\chi$;
\item $f_1(s)=L(s,\pi\times\pi')$ and $f_2(s)=L(s-2it_\chi,\tilde{\pi}\times\tilde{\pi}')$ for $s_0=1+it_\chi$.
\end{itemize}
These three cases correspond to the three lines in \eqref{eqn:D_def}, with $f_1(s),\dotsc,f_{m}(s)$ occurring as $m$ factors on the relevant line.  We note that regardless of the value of $s_0$, $\Gamma(s-\beta)$ is always a factor of $f_0(s)$.

We now apply \cref{lem:residue} in conjunction with \cref{lem:Li1} and \eqref{eqn:BH}. The functions $f_0(s),\dotsc,f_m(s)$ defined above are holomorphic in the open disk $|s-s_0|<\min(1-\beta,|t_\chi|)$. Moreover,
\[
|f_j(s_0)|=|L(1,\pi\times(\pi'\otimes\chi))|,\qquad j\in\{1,\dots,m\}.
\]
We are finished upon noting that, at the point $s_0$, the $k$-th derivative of $s\mapsto x^{s-\beta}$ is bounded by $x^{\epsilon/2}(\log x)^k$, while integration by parts yields
\begin{align*}
\left|\Gamma^{(k)}(s_0-\beta)\right|
&\leq\int_0^\infty r^{-\beta}|\log r|^k e^{-r}\,dr\\
&\leq\int_0^1 r^{-\beta}(-\log r)^k\,dr+\int_1^\infty(r-1)^k e^{-r}\,dr\\
&=\frac{k!}{(1-\beta)^{k+1}}+\frac{k!}{e}.
\end{align*}

Second, assume that exactly one of $\pi\otimes\chi^*=\pi$ and $\pi'\otimes\chi^*=\pi'$ holds true. Then $|t_{\chi}|\geq Q^{-\epsilon/8}$ by hypothesis, and $\mathcal{S}=\{1,1-it_\chi,1+it_\chi\}$. For each choice of $s_0\in\mathcal{S}$, let $m$ be the order of the pole of $D(s)$ at $s=s_0$. Specifically, $m=4$ for $s_0=1$, and $m=1$ for $s_0=1\pm it_\chi$. Consider the decomposition \eqref{eqn:decomposition}, where
\begin{itemize}
\item $f_1(s)=f_2(s)=L(s+it_\chi,\pi\times\pi')$ and $f_3(s)=f_4(s)=L(s-it_\chi,\tilde{\pi}\times\tilde{\pi}')$ for $s_0=1$;
\item $f_1(s)=L(s,\tilde\pi\times\tilde\pi')$ for $s_0=1-it_\chi$;
\item $f_1(s)=L(s,\pi\times\pi')$ for $s_0=1+it_\chi$.
\end{itemize}
From here, we proceed exactly as in the previous case.

Finally, assume that $\pi\otimes\chi^*\neq\pi$ and $\pi'\otimes\chi^*\neq\pi'$. Then $\mathcal{S}=\{1\}$ by hypothesis. The order of the pole of $D(s)$ at $s=1$ is 4. We consider the decomposition \eqref{eqn:decomposition}, where 
\[
f_1(s)=f_2(s)=L(s,\pi\times(\pi'\otimes\chi))\quad\text{and}\quad f_3(s)=f_4(s)=L(s,\tilde\pi\times(\tilde\pi'\otimes\bar{\chi})).
\]
These four factors occur on the first line of \eqref{eqn:D_def}, and we finish as in the other cases.
\end{proof}

\begin{proof}[Proof of \cref{prop:P1}]
We employ \cref{thm:standard_region} and Lemmata~\ref{tpiformula}, \ref{lem:Li1}, \ref{lem:nonneg}, \ref{lem:GHL}, and \ref{lem:residue_bounds} to prove \cref{prop:P1}.  First, we deduce \eqref{ZFR_beta} when $L(s,\pi\times(\pi'\otimes\chi^2))$ is entire, in which case all three $L$-functions in \eqref{eqn:Lfns} are entire by the initial assumptions.  Let $D(s)=D(s;\pi,\pi',\chi)$ be as in \eqref{eqn:D_def}, let $\mathcal{S}$ be the set of its poles, and let $x>1$ be a parameter to be chosen later. By assumption, there exists $\beta\in[1-\epsilon/16,1)$ such that $L(\beta,\pi\times\pi')=0$. Hence $D(\beta)=0$, and $\mathcal{S}$ is also the set of poles of $D(s)x^{s-\beta}\Gamma(s-\beta)$ in the half-plane $\Re(s)>0$.  We shall use this below.

If $\lambda_{D}(\ka)$ is the $\ka$-th Dirichlet coefficient of $D(s)$, then $\lambda_{D}(\ka)\geq 0$ by \cref{lem:nonneg}. Since $\lambda_{D}(\cO_F)=1$, we have by the residue theorem
\begin{align*}
\frac{1}{e}\leq \sum_{\ka}\frac{\lambda_{D}(\ka)}{\N\ka^{\beta}}e^{-\frac{\N\ka}{x}}&=\frac{1}{2\pi i}\int_{1-i\infty}^{1+i\infty}D(s+\beta)x^s\Gamma(s)\,ds\\
&=\sum_{s_0\in\mathcal{S}}\mathop{\mathrm{Res}}_{s=s_0}D(s)x^{s-\beta}\Gamma(s-\beta)
+\frac{1}{2\pi i}\int_{1/2-i\infty}^{1/2+i\infty}D(s)x^{s-\beta}\Gamma(s-\beta)\,ds.
\end{align*}
We estimate the sum over $\mathcal{S}$ by \cref{lem:residue_bounds} (with $\epsilon$ replaced by $\epsilon/8$), and the last integral by
\cref{lem:Li1} (with $\epsilon$ replaced by $\epsilon/32$) combined with \eqref{eqn:BH} and Stirling's formula. We conclude that
\begin{equation}
\label{eqn:previousbound}
1\ll_{n,n',[F:\Q],\epsilon} \left((1-\beta)^{-4}|L(1,\pi\times(\pi'\otimes\chi))|+Qx^{-1/2}\right)(Qx)^{\epsilon/8}.	
\end{equation}

At this point, we choose
\[
x = \max\left\{1,(1-\beta)^{8}Q^{2}|L(1,\pi\times(\pi'\otimes\chi))|^{-2}\right\}.
\]
If $x=1$, then
\[
|L(1,\pi\times(\pi'\otimes\chi))|\geq (1-\beta)^4 Q > (1-\beta)^4 Q^{-\epsilon/2}.
\]
Otherwise, $x=(1-\beta)^{8} Q^{2}|L(1,\pi\times(\pi'\otimes\chi))|^{-2}>1$, so \eqref{eqn:previousbound} yields
\[
|L(1,\pi\times(\pi'\otimes\chi))|\gg_{n,n',[F:\Q],\epsilon} (1-\beta)^{4}Q^{-3\epsilon/(8-2\epsilon)}>(1-\beta)^{4}Q^{-\epsilon/2}.
\]
Either way, we have established the lower bound
\begin{equation}
\label{eqn:L1_lower}
|L(1,\pi\times(\pi'\otimes\chi))|\gg_{n,n',[F:\Q],\epsilon}(1-\beta)^{4}Q^{-\epsilon/2}.
\end{equation}

We apply \cref{lem:GHL} with $\epsilon$ replaced by $\epsilon/64$, so we use \eqref{eqn:GHL_interval} with $\epsilon 2^{-14}Q^{-\epsilon/64}$ in place of $\epsilon 2^{-8}Q^{-\epsilon}$.  We can assume that $L(s,\pi\times(\pi'\otimes\chi))$ has a zero in the interval $[1-\epsilon 2^{-14}Q^{-\epsilon/64},1)$, for otherwise the conclusion \eqref{ZFR_beta} is clear. Let $\beta_\chi$ denote the largest such zero.  Since $\beta_\chi$ is real, it is also a zero of $L(s,\tilde\pi\times(\tilde\pi'\otimes\bar{\chi}))$ per \eqref{eqn:conjugate_symmetry}. Therefore, \cref{lem:GHL} ensures that $\beta$ lies outside the said interval:
\begin{equation}
\label{eqn:beta_lower}
1-\beta\gg_\epsilon Q^{-\epsilon/64}.
\end{equation}
We combine \eqref{eqn:L1_lower} and \eqref{eqn:beta_lower} with the straightforward inequality
\begin{equation}
\label{eqn:MVT}
|L(1,\pi\times(\pi'\otimes\chi))|\leq(1-\beta_\chi)\sup_{\sigma\in[\beta_\chi,1]}|L'(\sigma,\pi\times(\pi'\otimes\chi))|.
\end{equation}
The left-hand side of \eqref{eqn:MVT} is bounded from below per \eqref{eqn:L1_lower} and \eqref{eqn:beta_lower}, and the right-hand side of \eqref{eqn:MVT} is $\ll_{n,n',[F:\Q],\epsilon}(1-\beta_\chi) Q^{\epsilon/3}$ per \cref{lem:Li1} and \eqref{eqn:BH}.  Upon solving for $\beta_\chi$, we conclude \eqref{ZFR_beta} in the case when $L(s,\pi\times(\pi'\otimes\chi^2))$ is entire.

Now we deduce \eqref{ZFR_beta} when $L(s,\pi\times(\pi'\otimes\chi^2))$ has a pole. In this case, there exists a unique $t\in\R$ such that $\pi'\otimes\chi^2=\tilde\pi\otimes|\cdot|^{it}$. We calculate $t=t_\pi+t_{\pi'}+2t_\chi$. Consider the Hecke character $\kappa=\bar\chi|\cdot|^{it}$, which satisfies $\tilde\pi\otimes\kappa=\pi'\otimes\chi$. The estimate
\begin{equation}\label{eq:ckappabound}
C(\kappa)\leq C(\chi)(|t|+3)^{[F:\Q]}<(C(\pi)C(\pi')C(\chi)^3)^{[F:\Q]}
\end{equation}
follows from \eqref{eq:tpibound} and \eqref{eqn:BH}.  Moreover, the $L$-function
\[L(s,\pi\times(\tilde\pi\otimes\kappa^2))=L(s+it,\pi\times\pi')\]
is entire by assumption. Therefore,  \cref{thm:standard_region} and \cref{conductorremark} yield that
\[
L(\sigma,\pi\times(\tilde\pi\otimes\kappa))\neq 0,\qquad
\sigma\geq 1-\frac{1}{903\log(C(\pi)^{2n} C(\kappa)^{n^2})}.
\]
Using \eqref{eq:ckappabound} and the identity $L(\sigma,\pi\times(\tilde\pi\otimes\kappa))=L(\sigma,\pi\times(\pi'\otimes\chi))$, we conclude \eqref{ZFR_beta} in the case when $L(s,\pi\times(\pi'\otimes\chi^2))$ has a pole.
\end{proof}

\begin{proof}[Proof of \cref{prop:P2}] The argument is similar to but simpler than the proof of \cref{prop:P1}, so we shall be somewhat brief. Instead of $D(s)$ as in \eqref{eqn:D_def}, we work with $H(s)$ as in \eqref{eqn:D_def2}. By assumption, there exists $\beta\in[1-\epsilon/16,1)$ such that $L(\beta,\pi\times\tilde\pi)=0$. Hence for any $x>1$, the only pole of $H(s)x^{s-\beta}\Gamma(s-\beta)$ in the half-plane $\Re(s)>0$ is $s=1$, which is a double pole.

By mimicking the residue calculation in the proof of \cref{lem:residue_bounds}, and applying also Lemmata~\ref{lem:Li1} and \ref{lem:nonneg}, we obtain the following analogue of \eqref{eqn:previousbound}:
\[
1\ll_{n,n',[F:\Q],\epsilon} \left((1-\beta)^{-2}|L(1,\pi\times(\pi'\otimes\chi))|+Q^{1/4}x^{-1/2}\right)(Qx)^{\epsilon/8}.	
\]
By optimizing the parameter $x>1$, we obtain the analogue of \eqref{eqn:L1_lower}:
\begin{equation}
\label{eqn:L1_lower2}
|L(1,\pi\times(\pi'\otimes\chi))|\gg_{n,n',[F:\Q],\epsilon}(1-\beta)^{2}Q^{-\epsilon/4}.
\end{equation}
We can assume that $L(s,\pi\times(\pi'\otimes\chi))$ has a zero in the interval
\begin{equation}
\label{niceinterval}
\sigma\geq 1-\frac{\min\{\epsilon,\Cr{GHLnew2}\}}{\log(C(\pi)C(\pi'\otimes\chi))},
\end{equation}
for otherwise the conclusion \eqref{ZFR_beta2} is clear. Let $\beta_\chi$ denote the largest such zero.  Since $\beta_\chi$ is real, it is also a zero of $L(s,\tilde\pi\times(\tilde\pi'\otimes\bar{\chi}))$ per \eqref{eqn:conjugate_symmetry}. Therefore, \cref{lem:GHL2} ensures that $\beta$ lies outside the interval \eqref{niceinterval}. Now \eqref{eqn:L1_lower2} combined with \cref{lem:Li1} and \eqref{eqn:BH} yields
\[
Q^{-\epsilon/2}\ll_{n,n',[F:\Q],\epsilon}|L(1,\pi\times(\pi'\otimes\chi))|\ll_{n,n',[F:\Q],\epsilon}(1-\beta_\chi) Q^{\epsilon/2}.
\]
Upon solving for $\beta_\chi$, we conclude \eqref{ZFR_beta2}.
\end{proof}

\section{Proof of \cref{thm:Tatuzawa}}

We employ \cref{lem:GHL2} and Propositions~\ref{prop:P1} \& \ref{prop:P2} to prove \cref{thm:Tatuzawa}. For ease of notation, we shall seek zero-free intervals of the form
\begin{equation}
\label{niceinterval2}
\sigma\geq 1-\Cl[abcon]{ZFR6}(C(\pi)C(\pi')C(\chi)^{2n+2n'})^{-\epsilon},
\end{equation}
where 
$\Cr{ZFR6}=\Cr{ZFR6}(n,n',[F:\Q],\epsilon)>0$ is an effectively computable constant. Without loss of generality, $\epsilon\in(0,1)$. Consider
\[
\epsilon'=\frac{\epsilon}{2(n+n')}.
\]

Assume first that $L(s,\pi\times\tilde\pi)$ has a zero in $[1-\epsilon'/16,1)$. Let $\Cr{ZFR6}>0$ be suitably small. If $L(s,\pi\times(\pi'\otimes\chi))$ is entire, then by \cref{prop:P2} and \eqref{eqn:Rdef}, it has no zero in the interval \eqref{niceinterval2}. If $L(s,\pi\times(\pi'\otimes\chi))$ has a pole, then there exists a unique $t\in\R$ such that $\pi'\otimes\chi=\tilde\pi\otimes|\cdot|^{it}$.
By \cite[Theorem~5.1]{HumphriesThorner}, or by the more general \cref{thm:standard_region}, the corresponding $L$-function
\[L(s,\pi\times(\pi'\otimes\chi))=L(s+it,\pi\times\tilde\pi)\]
has no zero in the interval \eqref{niceinterval2} when $t\neq 0$, while it has at most one zero (necessarily simple) in the same interval when $t=0$. So we proved \cref{thm:Tatuzawa} in this case.

Assume now that $L(s,\pi\times\tilde\pi)$ has no zero in $[1-\epsilon'/16,1)$. Let $\Cr{ZFR6}>0$ be suitably small. If $L(s,\pi\times(\pi'\otimes\chi))$ has a pole, then there exists a unique $t\in\R$ such that $\pi'\otimes\chi=\tilde\pi\otimes|\cdot|^{it}$.
By \cite[Theorem~5.1]{HumphriesThorner}, or by the more general \cref{thm:standard_region}, coupled with our initial assumption, the corresponding $L$-function
\[L(s,\pi\times(\pi'\otimes\chi))=L(s+it,\pi\times\tilde\pi)\]
has no zero in the interval \eqref{niceinterval2}. Hence it suffices to show that among all the \emph{entire} twists $L(s,\pi\times(\pi'\otimes\chi))$, only one can have a zero in the interval \eqref{niceinterval2}, and even this exceptional twist can only have one zero (necessarily simple) in the same interval.

From now on, we restrict to $\chi$ lying in
\[
S=\left\{\chi\in\mathfrak{F}_1:\text{$L(s,\pi\times(\pi'\otimes\chi))$ is entire}\right\}.
\]
This set is nonempty, of course. As a first step towards our goal, we prove the existence of a constant $0<\Cr{ZFR6}\leq\epsilon'/16$ and a character $\lambda=\lambda_{\pi,\pi',\epsilon}\in S$ such that for all $\chi\in S$, one of the following two statements is true.
\begin{enumerate}
	\item $L(s,\pi\times(\pi'\otimes\chi))=L(s+it,\pi\times(\pi'\otimes\lambda))$ 
	with $t=t_\chi-t_\lambda$ of absolute value less than $(C(\pi)C(\pi'\otimes\chi)C(\pi'\otimes\lambda))^{-\epsilon/128}$.
	\item $L(\sigma,\pi\times(\pi'\otimes\chi))\neq 0$ in the interval \eqref{niceinterval2}.
\end{enumerate}
If, for all $\chi\in S$, the $L$-function $L(s,\pi\times(\pi'\otimes\chi))$ has no zero in $[1-\epsilon'/16,1)$, then we are finished. Otherwise, by 
a straightforward compactness argument and Shahidi's nonvanishing theorem \cite[Theorem~5.2]{Shahidi}, there exists $\lambda\in S$ with $C(\lambda)$ minimal such that $L(s,\pi\times(\pi'\otimes\lambda))$ has a zero in $[1-\epsilon'/16,1)$.

We fix $\lambda\in S$ as above, and we consider an arbitrary $\chi\in S$. If $C(\chi)<C(\lambda)$, then $L(s,\pi\times(\pi'\otimes\chi))$ does not vanish in $[1-\epsilon'/16,1)$ by the minimality of $C(\lambda)$, hence item (2) holds. From now on, we assume that $C(\chi)\geq C(\lambda)$, and we define
\[
\pi''=\pi'\otimes\lambda,\qquad \chi'=\chi\bar{\lambda}.
\]
The $L$-function $L(s,\pi\times\pi'')$ is entire and has a zero in $[1-\epsilon'/16,1)$. Moreover, the $L$-function
\begin{equation}\label{eqn:newtwist}
L(s,\pi\times(\pi'\otimes\chi))=L(s,\pi\times(\pi''\otimes\chi'))
\end{equation}
is entire. Hence, we shall apply \cref{prop:P1} with $(\pi,\pi'',\chi',\epsilon')$ in the role of $(\pi,\pi',\chi,\epsilon)$. Accordingly, we work with the following variant of \eqref{eqn:Rdef}:
\[
Q=Q(\pi,\pi'',\chi')=(C(\pi)C(\pi''))^{2(n+n')}C(\chi')^{(n+n')^2}.
\]
It follows from  $C(\lambda)\leq C(\chi)$ and \eqref{eqn:BH} that
\begin{equation}
\label{eqn:Qbound}
(C(\pi)C(\pi'\otimes\lambda)C(\pi'\otimes\chi))^{n+n'}<Q<(C(\pi)C(\pi'))^{2(n+n')}C(\chi)^{4(n+n')^2}.
\end{equation}

If the twist equivalence condition of \cref{prop:P1} is satisfied for $(\pi,\pi'',\chi')$ in place of $(\pi,\pi',\chi)$, and the constant $0<\Cr{ZFR6}\leq\epsilon'/16$ is suitably small, then by \eqref{ZFR_beta}, \eqref{eqn:newtwist}, and \eqref{eqn:Qbound}, we see that item (2) holds. Otherwise,
the relation
\[
\pi\otimes\chi^*=\pi\otimes\lambda^*\qquad\text{or}\qquad\pi'\otimes\chi^*=\pi'\otimes\lambda^*
\]
and the bound 
\[
|t_\chi-t_{\lambda}|<Q^{-\epsilon'/64}<(C(\pi)C(\pi'\otimes\chi)C(\pi'\otimes\lambda))^{-\epsilon/128}
\]
are simultaneously true. Therefore, by \eqref{eqn:Qbound} and the identity
\begin{align*}
L(s,\pi\times(\pi'\otimes\chi))
&=L(s+it_\chi,\pi\times(\pi'\otimes\chi^*))\\
&=L(s+it_\chi,\pi\times(\pi'\otimes\lambda^*))\\
&=L(s+it_\chi-it_{\lambda},\pi\times(\pi'\otimes\lambda)),
\end{align*}
item (1) holds. 

With the constant $0<\Cr{ZFR6}\leq\epsilon'/16$ and the character $\lambda\in S$ in hand, we can finish the proof of \cref{thm:Tatuzawa}. By applying \eqref{eqn:D_def2} and \cref{lem:GHL2} with $\lambda$ in the role of $\chi$, we infer that $L(s,\pi\times(\pi'\otimes\lambda))$ has at most one zero (counted with multiplicity) in the region
\begin{equation}
\label{niceregion2}
\sigma\geq 1-\frac{\Cr{GHLnew2}}{\log(C(\pi)C(\pi'\otimes\lambda))}
\qquad\text{and}\qquad
|t|\leq\frac{\Cr{GHLnew2}}{\log(C(\pi)C(\pi'\otimes\lambda))}.
\end{equation}
On the other hand, if we define the constant $\Cl[abcon]{ZFR2}=\Cr{ZFR2}(n,n',\epsilon)>0$ by the equation 
$(256/\epsilon)\Cr{ZFR2}^{-\epsilon/256}=\Cr{GHLnew2}$, then for all $\chi\in S$ satisfying $C(\pi)C(\pi'\otimes\chi)\geq\Cr{ZFR2}$, we have that
\begin{equation}
\label{epslog1}
(C(\pi)C(\pi'\otimes\chi)C(\pi'\otimes\lambda))^{-\epsilon/128}\leq\Cr{GHLnew2}/\log(C(\pi)C(\pi'\otimes\lambda)).
\end{equation}
By Brumley's effective zero-free region \eqref{eqn:Brumley_ZFR}, we can lower the constant $\Cr{ZFR6}>0$ so that item (2) holds whenever $C(\pi)C(\pi'\otimes\chi)<\Cr{ZFR2}$. As observed earlier, item (2) also holds whenever $C(\chi)<C(\lambda)$. By further lowering the constant $\Cr{ZFR6}>0$ and applying \eqref{eqn:BH}, we can achieve that
\begin{equation}
\label{epslog2}
\Cr{ZFR6}(C(\pi)C(\pi')C(\lambda)^{2n+2n'})^{-\epsilon}\leq\Cr{GHLnew2}/\log(C(\pi)C(\pi'\otimes\lambda)).
\end{equation}

Now let $\chi\in S$ violate item (2). Then item (1) holds along with the bounds $C(\chi)\geq C(\lambda)$ and
\eqref{epslog1}--\eqref{epslog2}, hence we obtain a zero $\sigma+it$ of $L(s,\pi\times(\pi'\otimes\lambda))$ in the region \eqref{niceregion2}, with $t=t_\chi-t_\lambda$. By the above, this zero must be simple and uniquely determined, so the twist $L(s,\pi\times(\pi'\otimes\chi))$ is uniquely determined as well. So among all the \emph{entire} twists $L(s,\pi\times(\pi'\otimes\chi))$, only one can have a zero in the interval \eqref{niceinterval2}, and even this exceptional twist can only have one zero (necessarily simple) in the same interval.
The proof of \cref{thm:Tatuzawa} is complete.

\bibliographystyle{abbrv}
\bibliography{HarcosThorner_Tatuzawa}

\end{document}